\documentclass[12pt,reqno]{article}

\usepackage[usenames]{color}
\usepackage{mathtools}
\usepackage{amsthm}
\usepackage{amsfonts}

\usepackage[colorlinks=true,
linkcolor=webgreen,
filecolor=webbrown,
citecolor=webgreen]{hyperref}

\definecolor{webgreen}{rgb}{0,.5,0}
\definecolor{webbrown}{rgb}{.6,0,0}

\newcommand{\ZZ}{\mathbb{Z}}

\newcommand{\Scb}{\ensuremath{S_{[c,b]}}}

\newcommand{\seqnum}[1]{\href{https://oeis.org/#1}{\underline{#1}}}

\begin{document}

\theoremstyle{plain}
\newtheorem{theorem}{Theorem}
\newtheorem{corollary}[theorem]{Corollary}
\newtheorem{lemma}[theorem]{Lemma}
\newtheorem{proposition}[theorem]{Proposition}

\theoremstyle{definition}
\newtheorem{definition}[theorem]{Definition}

\begin{center}
\vskip 1cm
{\LARGE\bf Fixed Points of Augmented Generalized \\
\vskip .1in
Happy Functions II: Oases and Mirages 
}
\vskip 1cm

\begin{minipage}[t]{0.5\textwidth}
\begin{center}
Breeanne Baker Swart \\
Department of Mathematical Sciences \\
The Citadel\\
171 Moultrie St. \\
Charleston, SC 29409 \\
USA \\
\href{mailto:breeanne.swart@citadel.edu}{\tt breeanne.swart@citadel.edu} \\
\ \\
Susan Crook\\
Division of Mathematics, Engineering and Computer Science\\
Loras College\\ 
1450 Alta Vista St.\\
Dubuque, IA 52001\\
USA\\
\href{mailto:susan.crook@loras.edu}{\tt susan.crook@loras.edu} \\
\ \\
Helen G. Grundman\\
Department of Mathematics\\
Bryn Mawr College\\
101 N. Merion Ave.\\
Bryn Mawr, PA 19010\\
USA\\
\href{mailto:grundman@brynmawr.edu}{\tt grundman@brynmawr.edu} \\
\end{center}
\end{minipage}
\begin{minipage}[t]{0.45\textwidth}
\begin{center}
Laura Hall-Seelig\\
Department of Mathematics\\
Merrimack College\\
315 Turnpike Street\\
North Andover, MA 01845\\
USA\\
\href{mailto:hallseeligl@merrimack.edu}{\tt hallseeligl@merrimack.edu} \\

\ \\
May Mei\\ 
Department of Mathematics and Computer Science\\
Denison University\\
100 West College Street\\
Granville, Ohio 43023\\
USA\\
\href{mailto:meim@denison.edu}{\tt meim@denison.edu} \\
\ \\
Laurie Zack \\
Department of Mathematical Sciences\\ 
High Point University\\
One University Parkway\\
High Point, NC 27268\\
USA\\
\href{mailto:lzack@highpoint.edu}{\tt lzack@highpoint.edu} \\
\end{center}
\end{minipage}
\end{center}
    
\vskip .2in

\begin{abstract}
An augmented generalized happy function $\Scb$ maps a positive integer to the sum of the squares of its base $b$ digits plus $c$. For $b\geq 2$ and $k \in \ZZ^+$, a \emph{$k$-desert base $b$} is a set of $k$ consecutive non-negative integers $c$ for each of which $\Scb$ has no fixed points. 
In this paper, we examine a complementary notion, a \emph{$k$-oasis base $b$}, which we define to be a set of $k$ consecutive non-negative integers $c$ for each of which $\Scb$ has a fixed point. In particular, after proving some basic properties of oases base $b$, we compute bounds on the lengths of oases base $b$ and compute the minimal examples of maximal length oases base $b$ for small values of $b$.  
\end{abstract}

\section{Introduction}\label{S:Introduction}

The concepts of happy number~\seqnum{A007770} and generalized happy number~\cite{genhappy,guy,honsberger} were generalized further~\cite{augment} by allowing for augmentation, as follows. 
\begin{definition}
For integers $c \geq 0$ and $b \geq 2$, the {\it augmented generalized happy function}, $\Scb:{\ZZ}^+ \rightarrow {\ZZ}^+$, is defined for $0\leq a_i \leq b-1$ and $a_n \neq 0$ by
\begin{equation*}\label{eq:definition}
\Scb\left(\sum_{i=0}^n b^i a_i  \right) =
c + \sum_{i=0}^n a_i^2.
\end{equation*}
The value $c$ is called the {\it augmenting constant} of $\Scb$.  A positive integer $a$ is called a {\it fixed point} of $\Scb$ if $\Scb(a) = a$.
A positive integer $a$ is a {\it happy number} if, for some $k\in \ZZ^+$, $S_{[0,10]}^k(a) = 1$.  
\end{definition}

The function $S_{[0,10]}$ is easily seen to have exactly one fixed point, while, depending on the values of $c$ and $b$, the function $\Scb$ may have zero, one, or multiple fixed points~\cite{desert}.  The case of zero fixed points is studied in Part I of this paper~\cite{desert}, in which Baker Swart et al.\ prove that for each $b \geq 2$, there exist arbitrarily long finite sequences of consecutive values of $c$ for which $\Scb$ has no fixed point.

In this work, we study the complementary case by considering sets of consecutive augmenting constants $c$ for which $\Scb$ has at least one fixed point and proving that, for each fixed $b$, the size of these sets is bounded.
In Section~\ref{S:OasisB}, we define the concept of $k$-oasis base $b$, determine some initial properties, and prove a bound, for each $b\geq 2$, on the lengths of oases base $b$.
In Section~\ref{S:MirageB}, we define the concept of a $k$-mirage base $b$, prove that the maximal length of mirages base $b$ bounds the maximal length of oases base $b$, and provide an algorithm for finding the maximal length of mirages base $b$.  Finally, we use the above to determine the maximal length of oases (and of mirages) base $b$, for all $b \leq 20$.

For later convenience, we note that if
$\sum_{i=0}^{n}b^ia_i$ is a fixed point of $\Scb$, then solving for $c$ yields that
\begin{equation}
 c = \sum_{i=0}^n(b^i-a_i)a_i. \label{eqn:cfixedpoint} 
\end{equation}
Thus, for a given base $b$ and an arbitrary positive integer $a$, there is at most one augmenting constant, $c$, such that $a$ is a fixed point of $\Scb$.

\section{Fixed point oases}\label{S:OasisB}

We begin by defining the key concept in this paper, a $k$-oasis base $b$, which is analogous to the concept of a $k$-desert base $b$, defined in Part I of this paper~\cite{desert}.

\begin{definition} For $b\geq 2$ and $k\in \ZZ^+$, a {\em $k$-oasis base $b$} is a set of $k$ consecutive non-negative integers $c$ for each of which $\Scb$ has at least one fixed point.  An {\em oasis base $b$} is a $k$-oasis base $b$, for some $k \geq 1$.  The \emph{length} of a $k$-oasis base $b$ is $k$. \end{definition}

Theorem~\ref{oases} provides some basic facts about the existence and lengths of oases base $b$ for different values of $b \geq 2$.

\begin{theorem}\label{oases}  
Let $b\geq 2$.
\begin{enumerate}
\item
There exists an oasis base $b$.
\item
If $b\geq 2$ is odd, then every $k$-oasis base $b$ has $k = 1$.
\item
If $b \geq 6$ is even, then there exists a $5$-oasis base $b$.
\end{enumerate}
\end{theorem}

\begin{proof}
First, for any base $b\geq 2$, since $S_{[0,b]}(1) = 1$, $\{0\}$ is a 1-oasis base $b$.

Next, let $b\geq 2$ be odd.
As shown by Baker Swart et al.~\cite[Lemma 2.3]{desert}, if $\Scb$ has a fixed point, then $c$ is even.  Part 2 of the theorem follows immediately.

Finally, let $b \geq 6$ be even and let $B = b/2$.  Set\\
\begin{tabular}{lll}
\hphantom{mmm}& $a_1={\left(B-2\right)b+1}$, 
\hphantom{mm} &$c_1=B^2-4$,\\
&$a_2={\left(B-1\right)b+2}$, & $c_2=B^2-3$, \\ 
&$a_3={Bb+2}$, &  $c_3=B^2-2$, \\
&$a_4={\left(B-1\right)b+1}$, & $c_4=B^2-1$,\\
&$a_5={Bb+1}$, and  &  $c_5=B^2$.
\end{tabular}\\
A direct calculation shows that for each $1 \leq i\leq 5$, 
$S_{[c_i,b]}(a_i) = a_i$.  Hence $\{c_1,c_2,c_3,c_4,c_5\}$ is a $5$-oasis base $b$.
\end{proof}

Given a $k$-oasis base $b$, it is easy to produce additional $k$-oases base $b$.

\begin{theorem}\label{t:infinite} Let $b \geq 2$ and $k \geq 1$.  If there exists a $k$-oasis base $b$, there exist infinitely many $k$-oases base $b$.
\end{theorem}

\begin{proof}
Fix $b\geq 2$ and let $\{c + j|1\leq j\leq k\}$ be a $k$-oasis base $b$.
For each $j$, $1\leq j \leq k$, let $a(j)$ denote a fixed point of $S_{[c+j,b]}$.
Fix $n \in \ZZ^+$ such that for each $j$, $a(j) < b^n$.  Then for each positive integer $t$ and for each $j$, $t b^n + a(j)$ is a fixed point of 
\[S_{[c + j + t b^n - S_{[0,b]}(t), b]}.\] 
Hence for each $t\in \ZZ^+$, 
$\{c + j + t b^n - S_{[0,b]}(t)|1\leq j\leq k\}$ is a $k$-oasis base $b$.
\end{proof}

The following theorems provide properties of fixed points associated with values of $c$ that are in the same oasis base $b$.  
 Theorem~\ref{t:samedigit} provides that unless two such fixed points are quite small, they must have the same number of digits, and Theorem~\ref{T:matchingdigits} says that few of the digits of the fixed points may differ.
These theorems are used in proving Theorem~\ref{t:maxoasis}, which gives a general upper bound for the length of an oasis base $b$, and again in Section~\ref{S:MirageB}, in proving the correctness of a method we provide for improving this bound.

\begin{theorem}\label{t:samedigit}
Let $c \geq 0$ and $b \geq 2$.  Let $a \in \ZZ^+$ have $n + 1 > 3$ digits and satisfy $\Scb({a}) = {a}$.  Then every fixed point with an augmenting constant in the same oasis base $b$ as $c$ has exactly $n+1$ digits.
\end{theorem}
 
We note that Theorem~\ref{t:samedigit} is optimal in that fixed points of two and three digits, respectively, can have augmenting constants in the same oasis.  For example, in base 16, the two-digit number $a = 85_{(16)}$ has augmenting constant $c = 44$ and the three-digit number $\hat{a} = 10(15)_{(16)} = 
1\cdot 16^2 + 15$ has augmenting constant $\hat{c} = 45$. Clearly $c$ and $\hat{c}$ are in the same oasis base 16.

\begin{proof}[Proof of Theorem~\ref{t:samedigit}]
Consider the collection of all values of $c$ in a fixed oasis base $b$, and the set of all fixed points of the happy functions base $b$ with those $c$s as augmenting constants.  Assume that two of these fixed points have different numbers of digits and at least one of them has more than 3 digits.
Then there must exist augmenting constants $\bar c$ and $\hat c$ in the oasis with $|\bar c - \hat c| \leq 1$, and fixed points $\bar a$ of $S_{[\bar c,b]}$ with $\bar n + 1 > 3$ digits and $\hat a$ of $S_{[\hat c,b]}$ with $\hat n + 1 \neq \bar n + 1$ digits.  We may assume without loss of generality that $\hat n < \bar n$.

In Part I of this paper~\cite[Theorem 4.2]{desert}, Baker Swart et al.\ showed that for $n \geq 2$, if $\Scb$ has a fixed point of $n + 1$ digits, then $m_{b,n} \leq c \leq M_{b,n}$, where the bounds, which are given explicitly in terms of their parameters, are sharp. They also showed~\cite[Lemma 4.3]{desert} that, for $n \geq 2$, $M_{b,n} + 1 < m_{b,n+1}$.

In the same work~\cite[Theorem 4.2]{desert}, the authors prove that, for $n \geq 2$, if $\Scb$ has a fixed point of $n + 1$ digits, then $m_{b,n} \leq c \leq M_{b,n}$, where the bounds, which are given explicitly in terms of their parameters, are sharp.  Additionally, they show~\cite[Lemma 4.3]{desert} that, for $n \geq 2$, $M_{b,n} + 1 < m_{b,n+1}$. It follows (using a simple induction argument) that if $\hat n \geq 2$, then $\hat c + 1\leq M_{b,\hat n} + 1 < m_{b,\bar n} \leq \bar c$. But this implies that $1 < \bar c - \hat c = |\bar c - \hat c|$, a contradiction.  

Thus, we have that $\hat n < 2$. Letting 
$\hat a = \sum_{i=0}^{\hat n}a_ib^i$, Equation~(\ref{eqn:cfixedpoint}) yields
\[\hat c = (b - a_1)a_1 + (1-a_0)a_0.\]
The largest possible value of $(b-a_1)a_1$ occurs when 
$a_1 = \lfloor b/2\rfloor$ and the  largest possible value of $(1-a_0)a_0$ is 0.  Thus,
\[\hat c = (b - a_1)a_1 + (1-a_0)a_0 \leq (b - \lfloor b/2\rfloor) \lfloor b/2\rfloor + 0 \leq b^2/4.\] 
Since $\bar a$ has more than three digits, $\bar n \geq 3$, and so~\cite[Theorem 4.2]{desert} implies that
\[\bar{c} \geq m_{b,\hat{n}} = b^{\hat{n}} - b^2 + 3b  - 3 \geq b^3 - b^2 + 3b - 3.\]
Combining these and the fact that $b \geq 2$ yields
\[1 \geq |\bar c - \hat c| \geq (b^3 - b^2 + 3b - 3) - b^2/4 > 1,\] a contradiction, completing the proof.
\end{proof}

\begin{theorem}\label{T:matchingdigits} Fix $b \geq 2$ and let $c$ and $\hat c$ be in the same oasis base $b$.  Let $\Scb(a) = a$ and $S_{[\hat{c},b]}(\hat{a})=\hat{a}$.  Then, letting $a_i$ and $\hat{a}_i$ denote the coefficients of $b^i$ in the base $b$ expansions of $a$ and $\hat{a}$, respectively, for each $i \geq 3$, $a_i = \hat{a}_i$. \end{theorem}

\begin{proof}
Suppose for a contradiction that there exists an $i\geq 3$ such that $a_i \neq \hat{a}_i$. 
Then, at least one of $a$ and $\hat{a}$ has more than three digits, and so, by Theorem~\ref{t:samedigit}, $a$ and $\hat{a}$ have the same number of digits, say $n + 1 > 3$.  We may assume, without loss of generality, that $|\hat{c} - c| \leq 1$.

Fix $j\geq 3$ maximal such that $a_j \neq \hat{a}_j$.
 Then using Equation~(\ref{eqn:cfixedpoint}), we have
\begin{align*}
|\hat{c} - c| 
& = \left|\left(\sum_{i=0}^n(b^i-\hat{a}_i)\hat{a}_i\right) - \left(\sum_{i=0}^n(b^i-a_i)a_i\right)\right|\\
& = \left| \sum_{i=0}^j\left((b^i-\hat{a}_i)\hat{a}_i - (b^i-a_i)a_i\right)\right|\\
& \geq \left| (b^j-\hat{a}_j)\hat{a}_j - (b^j - a_j)a_j \right| -
\sum_{i=0}^{j-1}\left|(b^i - \hat{a}_i)\hat{a}_i - (b^i-a_i)a_i\right|.
\end{align*}

For $i\geq 2$ and $0 \leq x \leq b-1$, the first derivative of the function $f(x) = (b^i - x)x$ is positive and decreasing.  Thus, the function is increasing at a decreasing rate over the domain.  Therefore, the smallest difference between the function values for two integer values of $x$ occurs when $x = b - 1$ and $x = b - 2$.  Similarly, the largest difference occurs when $x = b - 1$ and $x = 0$.  It follows that
\begin{align*}
|\hat{c} - c| 
\geq{}&\left((b^j - (b - 1))(b - 1) - (b^j - (b - 2))(b - 2)\right) \\
& - \sum_{i=2}^{j-1}\left((b^i - (b - 1))(b - 1) - (b^i - 0)(0)\right) \\
& - \sum_{i=0}^{1}\left|(b^i - \hat{a}_i)\hat{a}_i - (b^i-a_i)a_i\right|\\
\geq{}& b^j - (b-1)^2 + (b-2)^2 - \sum^{j-1}_{i=2} b^i(b-1) + (j - 2)(b - 1)^2 \\
& - \left|\left(b-\frac{b}{2}\right)\left(\frac{b}{2}\right) - (b - 0)(0)\right| \\
& -|(1 - (b - 1))(b - 1) - (1 - 0)(0)|\\
={}& b^2 + (b - 2)^2 + (j - 3)(b - 1)^2 - \frac{b^2}{4} - (b - 1)(b - 2).
\end{align*}
Since $j\geq 3$ and $b \geq 2$, this implies that
\[1 \geq |\hat{c} - c| \geq \frac{3}{4}b^2 - b + 2 > 2,\] a contradiction.
Thus no such $i \geq 3$ exists, as desired.
\end{proof}

\begin{theorem}\label{t:maxoasis}
Let $b \geq 2$.  If there exists a $k$-oasis base $b$, then 
\begin{equation*}
k \leq \frac{b^3}{2} +\frac{b^2}{2} - b.
\end{equation*}
\end{theorem}

\begin{proof}
If $b$ is odd, then, by Theorem~\ref{oases}, every oasis base $b$ has length 1, which is less than $b^3/2 + b^2/2 - b$.  So we assume that $b$ is even.

Let $\{c + j | 1 \leq j \leq k\}$ be a $k$-oasis base $b$.  For each $j$, $1 \leq j \leq k$, let $a(j)$ be a fixed point of $S_{[c+j,b]}$. By Theorem~\ref{T:matchingdigits}, the fixed points, $a(j)$, differ in, at most, the rightmost three digits.  Since each fixed point corresponds to exactly one augmenting constant, this implies an initial bound: $k \leq b^3$.  But we can improve on this.

Baker Swart et al.~\cite[Theorem 2.1]{desert} prove that if $a(j)$ is a multiple of $b$, then $a(j) + 1$ is also a fixed point of $S_{[c+j,b]}$.  Thus, substituting $a(j) + 1$ for $a(j)$, if necessary, we may assume that none of the $a(j)$ is a multiple of $b$.  This leaves us with $b-1$ possible rightmost digits.

Similarly, if $a(j)$ has second rightmost digit equal to $d \neq 0$, then the number obtained by replacing that digit with the digit $b-d$ is another fixed point of $S_{[c+j,b]}$~\cite[Lemma 2.2]{desert}.  Thus we may assume that none of the $a_i$ have second rightmost digit greater than $b/2$.  This leaves us with $(b+2)/2$ possible second rightmost digits.  

So, the number of possible values of the rightmost three digits of the $a(j)$s is 
\[(b)\left(\frac{b+2}{2}\right)(b - 1).\]
Since, for each of the $k$ augmenting constants, $c + j$, there is a distinct fixed point, $a(j)$, we conclude that the number of augmenting constants in the oasis is
\begin{equation*}
k \leq (b)\left(\frac{b+2}{2}\right)(b - 1) = \frac{b^3}{2}+\frac{b^2}{2}-b,
\end{equation*}
as desired.
\end{proof}

\section{Maximal lengths of oases base \texorpdfstring{$b$}{b}}
\label{S:MirageB}

In the previous section, we determined a general formula for an upper bound for the length of an oasis base $b$, for $b\geq 2$.  In this section, we present an algorithm for determining a new bound on this length, which, in many cases, can be shown to provide the precise maximal length of oases base $b$.

 From Theorem~\ref{oases}, we know that if $b$ is odd, every oasis base $b$ has length 1. For bounding the oasis lengths for even bases, we introduce the concept of a mirage base $b$.  
For convenience, we extend the domain of the function $S_{[0,b]}$ for $b \geq 2$ to include 0 by defining $S_{[0,b]}(0) = 0$.  

\begin{definition} For $b\geq 2$ and $k \in \ZZ^+$, a {\em $k$-mirage base $b$} is a set of $k$ consecutive integers $\{d_1,..., d_k\}$, such that for each $1\leq i \leq k$, $d_i = r_i - S_{[0,b]}(r_i)$ with $r_i$ a non-negative integer of at most three digits.  A {\em mirage base $b$} is a $k$-mirage base $b$ for some $k \geq 1$. \end{definition}  

A mirage base $b$ may or may not be an oasis base $b$.  
We first provide a large class of mirages that actually are oases.

\begin{lemma}\label{l:mirageisoasis}
If a $k$-mirage base $b$ contains only positive integers, then it is a $k$-oasis base $b$. 
\end{lemma}

\begin{proof}
Given a $k$-mirage containing only positive integers, using the notation in the definition of $k$-mirage base $b$, for each $1\leq i \leq k$, we have $S_{[d_i,b]}(r_i) = d_i + S_{[0,b]}(r_i) = r_i$.  Thus $r_i$ is a fixed point of $S_{[d_i,b]}$, and so $\{d_1,..., d_k\}$ is a $k$-oasis base $b$.
\end{proof}

Of course, not all mirages base $b$ are oases base $b$.  For example, 
$-4 = 16 - S_{[0,6]}(16)$, $-3 = 22 - S_{[0,6]}(22)$, $-2 = 2 - S_{[0,6]}(2)$, $-1 = 9 - S_{[0,6]}(9)$, and $0 = 1 - S_{[0,6]}(1)$, implying that $\{-4,-3,-2,-1,0\}$ is a 5-mirage base 6, though not an oasis base 6.

We next show that given a $k$-oasis base $b$, there must exist a $k$-mirage base $b$.

\begin{theorem}\label{t:mirage}
Given $b\geq 2$ and $k\in \ZZ^+$, if there exists a $k$-oasis base $b$, then there exists a $k$-mirage base $b$.
\end{theorem}
 
\begin{proof}
Let $\mathcal O = \{c + j| 1\leq j\leq k\}$ be a $k$-oasis base $b$, and for each $1\leq j \leq k$, let $a(j) \in \ZZ^+$ be a fixed point of $S_{[c +j,b]}$.  

First, consider the case in which each $a(j)$ has 3 or fewer digits.  Then for $1\leq j \leq k$, $a(j) = S_{[c+j,b]}(a(j)) = (c + j) + S_{[0,b]}(a(j))$, and so $c + j = a(j) - S_{[0,b]}(a(j))$.  Thus $\mathcal O$ is a $k$-mirage base $b$, and we are done.

Next, consider the case in which, for at least one value of $j$, $a(j)$ has more than 3 digits. Then, by Theorem~\ref{t:samedigit}, all of the $a(j)$ have the same number of digits, say $n + 1 > 3$.  For each $1\leq j\leq k$, let $0\leq a(j)_i \leq b-1$, such that 
\[a(j) = \sum_{i=0}^n a(j)_ib^i,\]  
and define \[r_j = \sum_{i=0}^2 a(j)_ib^i.\] 
By Theorem~\ref{T:matchingdigits}, for each $1\leq j\leq k$ and $i \geq 3$, $a(j)_i = a(1)_i$.  Thus, for each $1\leq j\leq k$, 
\begin{equation} \label{eq:C}
a(j) = r_j + \sum_{i=3}^n a(j)_ib^i = r_j + \sum_{j=3}^n a(1)_ib^i = r_j + a(1) - r_1.
\end{equation}
Further, since $a(j)$ is a fixed point of $S_{[c+j,b]}$, we have that
\begin{equation}\label{eq:C-hat}
\begin{split}
a(j) &= S_{[c+j,b]}(a(j)) = (c + j) + S_{[0,b]}(r_j) + \sum_{i=3}^n a(j)_i^2   \\
&= (c + j) + S_{[0,b]}(r_j) + \sum_{i=3}^n a(1)_i^2. 
\end{split}
\end{equation} 
Thus, using equations~(\ref{eq:C}) and~(\ref{eq:C-hat}), for each $1\leq j \leq k$,
\begin{equation} \label{eq:consec}
r_j - S_{[0,b]}(r_j) = \left(r_1 - a(1) + \sum_{i=3}^n a(1)_i^2 + c\right) + j.
\end{equation} 
Since the only value on the right-hand-side of Equation~(\ref{eq:consec}) dependent on $j$ is $j$ itself,  
\[\left\{r_j - S_{[0,b]}(r_j) | 1\leq j \leq k\right\}\]
is a set of consecutive integers and thus is a $k$-mirage base $b$.
\end{proof}

\begin{table}[b!]
\begin{center}
{\small
\begin{tabular}{|c|c|c|}\hline 
Base&Length&Minimal maximal length oasis  \\
\hline
&&Smallest fixed points\\
\hline\hline
2& 2 & $\{3,4\}$\\
\hline
& &  ${4,6}$\\
\hline\hline
4& 6 
& $\{28,29,30,31,32,33\}$ \\
\hline
&
& ${32, 38, 42, 36, 40, 51}$\\
\hline\hline
6& 5  
 & $\{5,6,7,8,9 \}$ \\
\hline
&
 &  ${6, 14, 20, 12, 18}$\\
\hline\hline
8 & 8 
 & $\{304,305,306,307,308,309,310,311\}$  \\
\hline
&
 &  ${347, 338, 391, 336, 346, 354, 344, 352}$\\
\hline\hline
10 & 8 
 & $\{487, 488, 489, 490, 491, 492, 493, 494\}$ \\
\hline
&
 &  ${544, 554, 522, 533, 520, 609, 543, 532 }$\\
\hline\hline
12 & 8 & \{172, 173, 174, 175, 176, 177, 178, 179\}  \\
\hline
& & {207, 194, 299, 192, 206, 218, 204, 216}  \\
\hline\hline
14 & 8
 & $\{421, 422, 423, 424, 425, 426, 427, 428\}$ \\
\hline
&
 &  ${434, 451, 601, 480, 494, 450, 465, 448}$\\
\hline\hline
16 & 8
 & $\{559, 560, 561, 562, 563, 564, 565, 566\}$ \\
\hline
&
 &  ${628, 644, 594, 611, 592, 799, 627, 610}$\\
\hline\hline
18 & 8
 & $\{1663, 1664, 1665, 1666, 1667, 1668, 1669, 1670\}$ \\
\hline
&
 &  ${1768, 1786, 1730, 1749, 1728, 1960, 1767, 1748}$\\
\hline\hline
20 & 9
 & $\{5124, 5125, 5126, 5127, 5128, 5129, 5130, 5131, 5132\}$ \\
\hline
&
 &  ${5383, 5362, 5699, 5360, 5382, 5402, 5380, 5400, 5617}$\\
 \hline
\end{tabular}
\caption{Minimal valued maximal length oases and smallest fixed points for small even bases. (Results are 
all given in base 10.) \label{max}}
}
\end{center}
\end{table}

It follows from Theorem~\ref{t:mirage} that the length of the longest mirage base $b$ bounds the length of the longest oasis base $b$. Since each element in a mirage is generated by a fixed point between $0$ and $b^3$ exclusive, the maximum length of a mirage base $b$ can be determined by a direct computer search.  Formalizing this algorithm: in order to determine, for some fixed $b \geq 2$, the maximal length of a mirage base $b$, and to check whether this is necessarily equal to the maximal length of an oasis base $b$, the following steps suffice.
\begin{enumerate}
\item
For each $0 < r < b^3$, compute $d = r - S_{[0,b]}(r)\in \ZZ$.
\item
Sort the values of $d$.
\item
Determine the length of the longest string of consecutive values of $d$.
\item
Check whether there is a longest string in which all of the values of $d$ are positive.
\end{enumerate}
The result of step 3 is the maximal length of a
mirage base $b$ and, therefore, a bound on the length of the maximal length oasis base $b$.  Each longest string found in step 3 is an example of a maximal length mirage base $b$.  If step 4 is answered in the affirmative, then this string of positive values of $d$ is also an example of a maximal length oasis base $b$.

We carry out this algorithm for all even bases $2 \leq b \leq 20$, in each case finding a maximal oasis base $b$.  We summarize the results in the following theorem.

\begin{theorem}
The maximal lengths of oases base $b$ for bases 2, 4, and 6, are 2, 6, and 5, respectively; the maximal length of oases base $b$ for bases 8, 10, 12, 14, 16, and 18 is 8; and the maximal length of oases base 20 is 9.
\end{theorem}

In Table~\ref{max} we provide, for each even base, $2\leq b \leq 20$, the minimal example of an oasis of maximal length, along with the smallest fixed point of the augmented happy function determined by each augmenting constant in the oasis.

\bigskip
\hrule
\bigskip

\noindent \textit{AMS 2010 Mathematics Subject Classification:} Primary 11A63.

\noindent \emph{Keywords:} happy number, fixed point, iteration.

\bigskip
\hrule
\bigskip
\end{document}